\documentclass[11pt,reqno]{amsart}
\usepackage[T1]{fontenc}
\usepackage{lmodern}
\usepackage{amsmath,amssymb,mathtools}
\usepackage{microtype}
\usepackage{enumitem}
\usepackage{needspace}
\usepackage[hidelinks,bookmarksnumbered]{hyperref}
\hypersetup{pdftitle={Clique partitions and bounded simplicial defect},pdfauthor={Obinna Okechukwu},pdfsubject={Extremal graph theory and clique partitions},pdfkeywords={clique partition, chordal graph, simplicial vertex, fractional packing, stability}}
\setlist[enumerate]{leftmargin=*,itemsep=2pt,topsep=4pt}
\newtheorem{theorem}{Theorem}[section]
\newtheorem{lemma}[theorem]{Lemma}
\newtheorem{proposition}[theorem]{Proposition}
\newtheorem{corollary}[theorem]{Corollary}
\theoremstyle{definition}

\newtheorem{conjecture}[theorem]{Conjecture}
\theoremstyle{remark}

\numberwithin{equation}{section}
\DeclareMathOperator{\cp}{cp}
\DeclareMathOperator{\rsd}{rsd}
\DeclareMathOperator{\MaxCut}{MaxCut}
\newcommand{\ind}{\overline K}
\newcommand{\doi}[1]{\href{https://doi.org/#1}{\nolinkurl{#1}}}
\title[Clique partitions and bounded simplicial defect]{Clique partitions and bounded simplicial defect}
\author{Obinna Okechukwu}
\thanks{Cambridge, MA, USA}
\date{September 2026}
\subjclass[2020]{05C35, 05C70, 05C17}
\keywords{Clique partition, chordal graph, simplicial vertex, fractional packing, stability}
\begin{document}
\begin{abstract}
The clique partition number of a graph is the minimum number of complete subgraphs whose edge sets partition its edge set. We study graphs in which, in every induced subgraph and outside every prescribed clique, some vertex has a neighbourhood that becomes a clique after deleting at most $s$ vertices. The case $s=0$ is exactly the class of chordal graphs. For each fixed $s$, we prove that the maximum clique partition number at all sufficiently large orders $n$ is
$\lfloor(n+s)(n+s+1)/6\rfloor-\binom{s+1}{2}$, and determine all equality graphs. The same expression is an upper bound up to an additive constant depending only on $s$ at every order. In particular, every chordal graph has clique partition number at most $n^2/6+n/6+O(1)$, answering a question of Erd\H{o}s, Ordman and Zalcstein. We also prove structural stability for sublinear defect and a sharp finite-order theorem for integer signed clique functionals. The proof combines signed fractional localization with an edge-disjoint triangle construction; only a qualitative fractional-packing approximation is required.
\end{abstract}
\maketitle

\section{Introduction}

For a finite simple graph $G$, let $\cp(G)$ be the minimum number of cliques of order at least two whose edge sets partition $E(G)$. Erd\H{o}s, Goodman and P\'osa proved that $\cp(G)\le\lfloor |G|^2/4\rfloor$, with a partition using only edges and triangles \cite[Theorem~4]{EGP}. Erd\H{o}s, Ordman and Zalcstein asked for the corresponding extremal bound on chordal graphs. They exhibited examples with $n^2/6+O(n)$ parts and proved an upper bound $(1/4-\varepsilon)n^2$ for some absolute $\varepsilon>0$ \cite{EOZ}. Their question, now listed as Erd\H{o}s Problem~81 \cite{Bloom}, asks whether $n^2/6+O(n)$ parts always suffice. We prove a stronger bound and extend it to a hereditary relaxation of chordality.

A vertex $v$ of a graph $H$ is \emph{$s$-simplicial} if
\[
 d_H(v)-\omega(H[N_H(v)])\le s,
\]
where the clique number of the empty graph is zero. Define the \emph{rooted simplicial defect} $\rsd(G)$ to be the least nonnegative integer $s$ such that, in every induced subgraph $H$ and outside every proper clique of $H$, there is an $s$-simplicial vertex. The empty clique is allowed. Equivalently, every noncomplete induced subgraph contains two nonadjacent $s$-simplicial vertices (Lemma~\ref{lem:rooted}). Dirac's simplicial-vertex theorem gives $\rsd(G)=0$ exactly for chordal graphs \cite{Dirac}.

Put
\begin{equation}\label{eq:profileQ}
 F_n=\left\lfloor\frac{n(n+1)}6\right\rfloor,
 \qquad Q_s(n)=F_{n+s}-\binom{s+1}{2}.
\end{equation}
The join of two vertex-disjoint graphs is denoted by $\vee$, and their disjoint union by $\sqcup$.

\begin{theorem}\label{thm:main}
For each integer $s\ge0$ there are constants $K_s\ge0$ and $N_s$ with the following properties. Every graph $G$ of order $n$ with $\rsd(G)\le s$ satisfies
\begin{equation}\label{eq:constantbound}
 \cp(G)\le Q_s(n)+K_s.
\end{equation}
For $n\ge N_s$, the maximum is exactly $Q_s(n)$. Equality at these orders holds precisely for the graphs
\begin{equation}\label{eq:extremal}
 \left(K_{k-s}\sqcup\ind_s\right)\vee\ind_{n-k},
 \qquad k\text{ a nearest integer to }\frac{2(n+s)+1}{6}.
\end{equation}
\end{theorem}

In particular, the extremal value for fixed $s$ is
\[
 \frac{n^2}{6}+\frac{2s+1}{6}n-\frac{s(s+1)}3+O(1).
\]
The coefficient of the linear term is therefore determined by the defect.

\begin{corollary}\label{cor:chordal}
Every chordal graph $G$ of order $n$ satisfies
\[
 \cp(G)\le\left\lfloor\frac{n(n+1)}6\right\rfloor+O(1).
\]
For all sufficiently large $n$, the maximum is $F_n$, and the equality graphs are exactly $K_k\vee\ind_{n-k}$ with $k$ a nearest integer to $(2n+1)/6$.
\end{corollary}

The graphs $K_k\vee\ind_{n-k}$ are the extremal examples already underlying \cite{EOZ}. Split graphs were studied separately by Wallis and Wu \cite{WallisWu} and by Chen, Erd\H{o}s and Ordman \cite{CEO}. The computational difficulty remains even within split graphs: deciding their clique partition number is NP-complete \cite{WallisWu}. Our conclusions are extremal theorems, not formulas for the optimum of every graph in the class.

The hereditary condition in Theorem~\ref{thm:main} does not impose a useful cut or a near-extremal form in advance. For example, $\rsd(K_m\vee H)=\rsd(H)$, so the class with defect at most one contains joins of arbitrarily large cliques with disjoint unions of cycles. The subclass $s=0$ is understood through perfect elimination orders \cite{Dirac,FulkersonGross}; the proof below uses orders ending in a prescribed clique. Generalized elimination orderings for hereditary graph classes were studied systematically by Aboulker, Charbit, Trotignon and Vu\v{s}kovi\'c \cite{ACTV}. After reversing their order convention, an unrooted $s$-defective order is an $\mathcal F_s$-elimination order for the family $\mathcal F_s$ of graphs $J$ with $|J|-\omega(J)>s$; our rooted condition requires such an order ending at every prescribed clique. Distance-based notions of simplicial elimination characterize graphs with bounded induced-cycle length \cite{Krithika}. Our definition instead counts deleted neighbours; cycles of arbitrary length already have rooted defect at most one, since every vertex of every induced subgraph has neighbourhood deficiency at most one.

We also obtain a stability statement in which the defect may grow with the order. For a clique $A\subseteq V(G)$, write
\begin{equation}\label{eq:holesA}
 D_A=|A|\,|V(G)\setminus A|-e_G(A,V(G)\setminus A).
\end{equation}

\begin{theorem}\label{thm:stabilityintro}
Let $|G_j|=n_j\to\infty$. If
\[
 \rsd(G_j)=o(n_j),\qquad
 \cp(G_j)\ge\frac{n_j^2}{6}-o(n_j^2),
\]
then there are cliques $A_j\subseteq V(G_j)$ such that
\[
 |A_j|=\frac{n_j}{3}+o(n_j),\qquad
 D_{A_j}+e(G_j-A_j)=o(n_j^2).
\]
There is also a function $\epsilon(n)\to0$, independent of the graph, such that every graph of order $n\ge8$ satisfies
\begin{equation}\label{eq:allgraph}
 \cp(G)\le\frac{(2n+1)^2}{24}+n\rsd(G)+\epsilon(n)n^2.
\end{equation}
\end{theorem}

The proof first localizes a signed fractional optimum at an actual clique. The approximation theorem of Haxell and R\"odl \cite{HR}, in the fixed-family form proved by Yuster \cite{Yuster}, is used only to obtain this localization from a large integral partition number. A separate construction then gives an exact partition near the extremal configuration. Galvin's list edge-colouring theorem \cite{Galvin} supplies one part of that construction. A bounded exceptional set is controlled by an incidence inequality and a finite optimization of its normalized neighbourhood sizes. Equality is obtained from the final finite partition count, rather than from the fractional approximation.

Section~\ref{sec:setup} gives the elimination facts and attaining examples. Section~\ref{sec:fractional} proves fractional localization, its packing conversion, and Theorem~\ref{thm:stabilityintro}. Sections~\ref{sec:construction} and \ref{sec:rigidity} construct the partitions and prove Theorem~\ref{thm:main}. Section~\ref{sec:integer} gives the finite signed theorem. The final section states the all-order integral conjecture and quantitative stability questions.

\section{Elimination and the extremal examples}\label{sec:setup}

All graphs are finite and simple. We write $|G|=|V(G)|$, use $G[U]$ for an induced subgraph, and set $e(U)=e(G[U])$ when the graph is fixed. For disjoint vertex sets $U,V$, the number of edges between them is $e(U,V)$. A clique partition of an edgeless graph has zero parts. All asymptotic statements below are along sequences whose orders tend to infinity; a parameter described as fixed is held constant along the sequence.

For an order $\pi$ of $V(G)$, let $N^+_\pi(v)$ be the neighbours appearing after $v$. The order is \emph{$s$-defective} if
\[
 |N^+_\pi(v)|-\omega(G[N^+_\pi(v)])\le s
 \qquad(v\in V(G)).
\]
An induced suborder has the same property: deleting a vertex from a graph decreases its order by one and its clique number by at most one.

\Needspace{9\baselineskip}
\begin{lemma}\label{lem:rooted}
The following conditions are equivalent for a graph $G$ and an integer $s\ge0$:
\begin{enumerate}[label=\textup{(\roman*)}]
\item $\rsd(G)\le s$;
\item every induced subgraph has an $s$-defective order ending in any prescribed clique;
\item every noncomplete induced subgraph has two nonadjacent $s$-simplicial vertices.
\end{enumerate}
The parameter is induced-hereditary, and every nonempty induced subgraph $H$ is $(\omega(H)+s-1)$-degenerate when $\rsd(G)\le s$.
\end{lemma}
\begin{proof}
Repeated deletion outside the prescribed clique proves (i)$\Rightarrow$(ii). The first exterior vertex of such an order proves the converse. If the set of all $s$-simplicial vertices in a noncomplete induced subgraph were a clique, it would be a proper root containing every such vertex, contradicting (i). This proves (iii). Conversely, a clique contains at most one endpoint of the nonadjacent pair in (iii), and every vertex of a complete graph is simplicial. Thus (iii) implies (i).

Induced heredity is part of (i). Applying it with the empty root gives a vertex of degree at most $s+\omega(H)-1$ in every induced subgraph of $H$, proving the degeneracy bound. In particular,
\begin{equation}\label{eq:chromatic}
 \chi(H)\le\omega(H)+s.\qedhere
\end{equation}
\end{proof}

The identification with chordal graphs follows from Dirac's theorem \cite{Dirac}. In the reverse direction, an induced cycle of length at least four has no simplicial vertex, so it cannot satisfy Lemma~\ref{lem:rooted} with $s=0$.

\begin{lemma}\label{lem:joins}
For every graph $H$,
\[
 \rsd(H)\le |H|-\omega(H),\qquad
 \rsd(K_m\vee H)=\rsd(H)\quad(m\ge0).
\]
\end{lemma}
\begin{proof}
A maximum clique leaves $|H|-\omega(H)$ vertices outside it. Its intersection with any neighbourhood is a clique, so every vertex has neighbourhood deficiency at most that number. Clique deficiency does not increase on taking induced subgraphs, proving the first assertion.

For the second assertion, induced heredity gives the lower bound. In an induced subgraph of the join, eliminate the $H$-vertices outside the root using a defective order ending in the $H$-part of the root, while retaining all complete-side vertices. The retained complete-side vertices increase the order and clique number of each forward neighbourhood by the same amount. The remaining graph is complete, and can be finished outside the root and then inside it.
\end{proof}

The following obstruction will replace common-neighbourhood arguments specific to chordal graphs.

\begin{lemma}\label{lem:forbidden}
A graph of rooted simplicial defect at most $s$ cannot contain $s+1$ vertex-disjoint nonedges whose endpoints have $s+2$ independent common neighbours.
\end{lemma}
\begin{proof}
Let $W$ be the $2s+2$ endpoints and let $P$ be the independent common-neighbour set of size $s+2$. Every clique in $G[W]$ omits an endpoint of each nonedge, so $|W|-\omega(G[W])\ge s+1$. In $G[W\cup P]$, retain a maximum clique of $G[W]$ as root. Each vertex of $P$ has deficiency at least $s+1$. At a vertex $x\in W$ outside the root, adjoining the independent set $P$ to $N_{G[W]}(x)$ increases the neighbourhood order by $|P|$ and its clique number by one. Its deficiency is therefore at least $|P|-1=s+1$. This contradicts the rooted condition.
\end{proof}

We use the following elementary edge-colouring facts.

\begin{lemma}\label{lem:colouring}
A simple $d$-degenerate graph of maximum degree $\Delta$ has a proper edge-colouring with at most $\Delta+d$ colours. If a graph has a proper $k$-edge-colouring, it has one in which the colour classes differ in size by at most one.
\end{lemma}
\begin{proof}
For the first assertion, remove a vertex of degree at most $d$, colour the residual by induction from the same palette, and restore its edges successively. At a restored edge, at most $(\Delta-1)+(d-1)$ colours are forbidden. The case of no edges is immediate.

For the second assertion, choose a proper colouring minimizing the sum of the squared colour-class sizes, allowing empty classes. If two classes differ in size by at least two, their alternating union has a path component with one more edge of the larger colour. Interchanging the colours on that component decreases the sum of squares, a contradiction.
\end{proof}

To record the attaining graphs, define
\[
 B_n(c)=c(n-c)-\binom c2,
 \qquad M_n=\frac{(2n+1)^2}{24}.
\]
The square completion
\begin{equation}\label{eq:quadratic}
 c(n-c)-\binom{c-s}{2}
 =M_{n+s}-\binom{s+1}{2}
 -\frac32\left(c-\frac{2(n+s)+1}{6}\right)^2
\end{equation}
has maximum $Q_s(n)$ over integer $c$. Indeed, the nearest-integer loss from $M_N$ is $1/24$ when $N\equiv0,2\pmod3$, and $3/8$ when $N\equiv1\pmod3$; in both cases the resulting value is $F_N$. For sufficiently large $n$ with $s$ fixed, every maximizing $c$ satisfies $s+1\le c\le n/2$.

\begin{lemma}\label{lem:attainment}
Let $k\ge s+1$ and $b\ge k$. The graph
\[
 G=(K_{k-s}\sqcup\ind_s)\vee\ind_b
\]
has $\rsd(G)=s$ and
\begin{equation}\label{eq:attainingcost}
 \cp(G)=kb-\binom{k-s}{2}.
\end{equation}
\end{lemma}
\begin{proof}
Call the $b$ independent vertices pages and the other $k$ vertices the core. With any prescribed clique retained, first delete the non-root pages. Their forward neighbourhoods have clique deficiency at most $s$. The remaining graph is a clique with an independent set joined through at most one retained page; it is chordal and can be eliminated outside the root. This proves $\rsd(G)\le s$. For the reverse inequality, retain the $(k-s)$-clique as root. Every page has deficiency $s$, and every exceptional core vertex has deficiency $b-1\ge s$.

A part containing a page and $j$ core vertices covers $j$ spokes and consumes $\binom j2$ core edges. Summing $j\le1+\binom j2$ gives $\cp(G)\ge kb-\binom{k-s}{2}$. Label the forced core clique by residues modulo $k-s$, and colour its edge $ij$ by $i+j$. Incident edges have different colours. Assign the colours to distinct pages, extend core edges to triangles through their assigned pages, and take all unused spokes singly. This attains the lower bound. The clique of order one has no edges to colour.
\end{proof}

\section{Signed fractional localization}\label{sec:fractional}

Fix an integer $L\ge3$. Let $\mathcal K_L(G)$ be the cliques of $G$ of orders from two to $L$. Write $q_L(G)$ for the minimum size of an edge partition using these cliques, and let
\begin{equation}\label{eq:LP}
 q_L^*(G)=\min\left\{\sum_{Q\in\mathcal K_L(G)}x_Q:
 x_Q\ge0,\quad \sum_{Q\ni e}x_Q=1\ (e\in E(G))\right\}.
\end{equation}
The notation $Q\ni e$ means that $e\in E(Q)$. The dual program is
\begin{equation}\label{eq:dual}
 q_L^*(G)=\max\left\{\sum_{e\in E(G)}z_e:
 \sum_{e\in E(Q)}z_e\le1\ (Q\in\mathcal K_L(G))\right\}.
\end{equation}
The coordinates $z_e$ are real and unrestricted in sign. The edge constraints imply $z_e\le1$. The primal is feasible using edges alone. An optimal dual exists; alternatively, coordinates below $2-\binom L2$ may be raised to that value without violating any clique constraint, so the dual maximum can be taken over a compact set.

\subsection{A maximal positive clique-star}

For an edge set $S$, define
\[
 Z(S)=\sum_{e\in S}z_e,
 \qquad N_z(S)=\sum_{e\in S}(-z_e)_+.
\]
The next estimate holds for every feasible dual vector, not only an optimum.

\begin{lemma}\label{lem:localize}
Let $n\ge8$, $L\ge4$, $\rsd(G)\le s$, and let $z$ be feasible in \eqref{eq:dual}. Put $W=Z(E(G))$. There are a vertex $v$ and a clique $C\subseteq N(v)$, possibly empty, such that, with
\[
 c=|C|,\quad \alpha=\sum_{u\in C}z_{vu},\quad
 N_{\rm out}=N_z(E(G)\setminus E(G[C])),
\]
we have $0\le\alpha\le c$ and
\begin{equation}\label{eq:rootineq}
 W+N_{\rm out}\le(n-c)(\alpha+s)+Z(E(G[C])).
\end{equation}
If $c\le(n+1)/2$, then
\begin{align}\label{eq:localdeficit}
 M_n+sn-W\ \ge\ &\frac32\left(c-\frac{2n+1}{6}\right)^2
 +(n-2c+1)(c-\alpha)+sc+N_{\rm out}.
\end{align}
If $c\ge(n+1)/2$, then
\begin{equation}\label{eq:largecore}
 W+N_{\rm out}\le\frac{25n^2}{192}+sn.
\end{equation}
In particular, $q_L^*(G)\le M_n+n\rsd(G)$.
\end{lemma}
\begin{proof}
Maximize $\sum_{u\in C}z_{vu}$ over vertices $v$ and cliques $C\subseteq N(v)$, allowing the empty clique. Remove any nonpositive entries from a maximizing clique. Thus its entries are positive unless $C$ is empty.

Take an $s$-defective order ending in $C$. At each exterior vertex the forward neighbourhood is a clique together with at most $s$ exceptional vertices. The sum of its positive weights on the clique is at most $\alpha$, since the positive neighbours still form a clique. Its positive exceptional weights sum to at most $s$. Every edge outside $E(G[C])$ is counted at its earlier endpoint. Adding negative mass to the signed sum leaves precisely the positive sum, and proves \eqref{eq:rootineq}.

Sum the triangle inequalities through the chosen vertex $v$ to obtain
\begin{equation}\label{eq:trianglestar}
 Z(E(G[C]))\le\binom c2-(c-1)\alpha.
\end{equation}
Writing $B_n(c)=c(n-c)-\binom c2$, equations \eqref{eq:rootineq}--\eqref{eq:trianglestar} give
\[
 W+N_{\rm out}
 \le B_n(c)-(n-2c+1)(c-\alpha)+s(n-c).
\]
The square completion \eqref{eq:quadratic} with $s=0$ proves \eqref{eq:localdeficit}.

In the other range $c\ge(n+1)/2$, we have $c\ge4$. Averaging all four-clique constraints inside $C$ yields
\[
 Z(E(G[C]))\le\frac16\binom c2.
\]
After subtracting $s(n-c)$, the right side of \eqref{eq:rootineq} is bounded by the minimum of
\[
 (n-2c+1)\alpha+\binom c2
 \quad\hbox{and}\quad
 (n-c)\alpha+\frac16\binom c2.
\]
The first is nonincreasing in $\alpha$, the second is nondecreasing, and they meet at $\alpha=5c/12$. Their minimum is therefore at most
\[
 \frac{c(5n-4c-1)}{12}
 \le\frac{(5n-1)^2}{192}\le\frac{25n^2}{192}.
\]
This proves \eqref{eq:largecore}. Both ranges imply the final assertion.
\end{proof}

\subsection{A cover estimate with signed weights}

A cover suffices for the next bound, provided repeated negative weights are counted. This avoids any assertion about an edge-deleted graph retaining its elimination property.

\begin{lemma}\label{lem:signedcover}
Let $H$ have at most $n$ vertices and admit an $s$-defective order. Suppose $z$ satisfies all clique inequalities of orders two through $L$. For every $\xi>0$,
\begin{equation}\label{eq:coverbound}
 Z(E(H))\le
 \left(\xi+\frac1{2(L-1)}\right)n^2+(s+3)n
 +\frac{L-2}{\xi}N_z(E(H)).
\end{equation}
\end{lemma}
\begin{proof}
At each vertex, cover its at most $s$ exceptional forward edges singly. Let $U$ be the clique part of its forward neighbourhood. If $|U|<\lceil\xi n\rceil+1$, cover those forward edges singly. Otherwise order $U$ uniformly at random, split it into consecutive groups of at most $L-1$ vertices, and adjoin the current vertex to each group. All resulting parts are cliques of orders at most $L$.

This covers each edge once at its earlier endpoint, possibly also at earlier vertices. The number of parts is at most
\[
 sn+n\lceil\xi n\rceil+\frac{e(H)}{L-1}+n
 \le\left(\xi+\frac1{2(L-1)}\right)n^2+(s+3)n.
\]
A fixed pair in a large $U$ lies in one group with probability at most
$(L-2)/(|U|-1)\le(L-2)/(\xi n)$. There are at most $n$ possible earlier vertices, so each edge has expected additional multiplicity at most $(L-2)/\xi$.

If $\ell_e\ge1$ is the multiplicity of an edge in this cover, summing its clique inequalities gives
\[
 \sum_e z_e
 \le\#\{\text{cover members}\}+\sum_e(\ell_e-1)(-z_e)_+.
\]
Taking expectations proves \eqref{eq:coverbound}.
\end{proof}

\begin{theorem}\label{thm:fracstability}
For every $\varepsilon>0$ there are an integer $L\ge4$, constants $\delta_0,\eta_0>0$, and an integer $n_0$ such that the following holds. If $|G|=n\ge n_0$,
\[
 \rsd(G)\le\delta_0 n,\qquad
 q_L^*(G)\ge(1/6-\eta_0)n^2,
\]
then $G$ has a clique $C$ for which
\[
 \bigl||C|-n/3\bigr|\le\varepsilon n,
 \qquad D_C+e(G-C)\le\varepsilon n^2.
\]
\end{theorem}
\begin{proof}
It suffices to consider $0<\varepsilon<1/100$. We give estimates that specify the order of all parameter choices. Take an optimal dual vector, put $s=\rsd(G)$, and use Lemma~\ref{lem:localize}. If $\delta_0+\eta_0<7/192$, the lower bound on its value excludes \eqref{eq:largecore}. Set
\[
 \varrho=M_n+sn-W.
\]
Then $0\le\varrho\le(\delta_0+\eta_0)n^2+n/6+1/24$. If this is at most $\gamma n^2$ for a sufficiently small $\gamma$, equation \eqref{eq:localdeficit} gives
\begin{equation}\label{eq:rhoestimates}
 \left|c-\frac{2n+1}{6}\right|\le\sqrt{2\varrho/3},
 \quad \frac n4\le c\le\frac{5n}{12},
 \quad c-\alpha\le\frac{6\varrho}{n},
 \quad N_{\rm out}\le\varrho.
\end{equation}
Let $R=V(G)\setminus C$, and for $x\in R$ put
\[
 b_x=c-\sum_{u\in C\cap N(x)}z_{xu},\qquad B=\sum_{x\in R}b_x.
\]
Each missing spoke contributes one to $B$ and each present spoke contributes $1-z_{xu}\ge0$, so $D_C\le B$. Decomposing the total signed value and using \eqref{eq:trianglestar} yields
\[
 W\le B_n(c)-B+(c-1)(c-\alpha)+Z(E(G[R])).
\]
Since $B_n(c)\le M_n$ and \eqref{eq:rhoestimates} holds,
\begin{equation}\label{eq:rowB}
 B\le7\varrho+Z(E(G[R])).
\end{equation}
Apply Lemma~\ref{lem:signedcover} to the induced graph $G[R]$. With $A_0=(L-2)/\xi$ this gives
\begin{equation}\label{eq:Bbound}
 \frac{B}{n^2}\le
 \xi+\frac1{2(L-1)}+\delta_0+\frac3n+(7+A_0)\gamma.
\end{equation}

There are at most $8B/c\le32B/n$ vertices with $b_x\ge c/8$, and at most $32B$ row edges incident with them. If adjacent rows $x,y$ both have $b_x,b_y<c/8$, the sum of their paired spoke deficiencies is less than $c/4$. Some $u\in C$ has paired deficiency less than $1/4$. Both spokes are present, their weights sum to more than $7/4$, and the triangle constraint gives $z_{xy}<-3/4$. There are at most $2\varrho$ such edges, by \eqref{eq:rhoestimates}. Hence
\begin{equation}\label{eq:structuralbudget}
 D_C+e(G[R])\le33B+2\varrho.
\end{equation}

Choose $\xi>0$ and then $L\ge4$ so that
$33(\xi+1/[2(L-1)])<\varepsilon/4$. Next choose $\gamma>0$ so small that
\[
 \sqrt{2\gamma/3}<\varepsilon/4,
 \qquad \bigl(33(7+A_0)+2\bigr)\gamma<\varepsilon/4,
 \qquad\gamma<10^{-3}.
\]
Choose $\delta_0,\eta_0>0$ with
$\delta_0+\eta_0<\min\{\gamma/2,7/192\}$ and $33\delta_0<\varepsilon/8$.
Finally choose $n_0$ large enough that $n/6+1/24\le\gamma n^2/2$, $99/n<\varepsilon/8$, and the constant $1/6$ in the first estimate of \eqref{eq:rhoestimates} is at most $\varepsilon n/4$. Equations \eqref{eq:Bbound}--\eqref{eq:structuralbudget} and \eqref{eq:rhoestimates} prove both conclusions.
\end{proof}

\subsection{From fractional to integral bounded-clique partitions}

For a fixed graph $J$ without isolated vertices, let $\nu_J(G)$ and $\nu_J^*(G)$ denote its integral and fractional edge-packing numbers. We use
\begin{equation}\label{eq:HR}
 \nu_J(G)\ge\nu_J^*(G)-o(n^2)
 \qquad(|G|=n),
\end{equation}
uniformly over $G$. This is the theorem of Haxell and R\"odl \cite{HR}; Yuster gives a shorter proof and the fixed-family extension \cite{Yuster}. The graph $J$ may be disconnected, and copies need not be induced. Isolated vertices of the host have no effect: apply the theorem on the nonisolated vertices when their number exceeds its threshold, and use the bounded total number of edges otherwise.

The saving from a $j$-clique is $\binom j2-1$, so unweighted rounding of each size separately would not ensure edge-disjointness between the rounded families. The next argument bundles the sizes before applying \eqref{eq:HR}.

\begin{lemma}\label{lem:rounding}
For every fixed $L\ge3$ and every $\zeta>0$, all sufficiently large graphs satisfy
\begin{equation}\label{eq:boundedrounding}
 q_L(G)\le q_L^*(G)+\zeta |G|^2.
\end{equation}
\end{lemma}
\begin{proof}
An exact fractional partition is equivalently a fractional packing of cliques of orders $3,\ldots,L$, with unused edge capacities assigned to singleton edges. Write $t_j$ for the total packing mass of $j$-cliques, and $w_j=\binom j2-1$ for its saving. Each $t_j\le n^2/6$.

Choose $\alpha>0$ with $\alpha\sum_{j=3}^Lw_j<\zeta/4$, put
$a_j=\lfloor t_j/(\alpha n^2)\rfloor$, and let $J$ be the disjoint union of $a_j$ copies of $K_j$ for each $j$. For fixed $L,\alpha$, only finitely many templates $J$ arise. If all $a_j=0$, the total saving is less than $\zeta n^2/4$, and edges alone suffice. Otherwise set
\[
 h=|J|,\qquad w_J=\sum_{j=3}^La_jw_j>0.
\]
We construct a fractional $J$-packing of mass at least $\alpha n^2-hn$, for $n$ large enough that this number is positive.

The total fractional mass of cliques meeting any fixed vertex is at most $(n-1)/2$: every such clique uses at least two of its incident edge capacities. Thus the mass meeting any set of at most $h$ vertices is less than $hn$. Suppose bundles of total mass $u<\alpha n^2-hn$ have been selected. For any required size $j$, the remaining mass is
\[
 t_j-a_ju>a_jhn\ge hn.
\]
We can therefore choose all the required component cliques successively, each disjoint in vertices from those already chosen in this bundle. Their union, with only their component edges, is a copy of $J$. Subtract the minimum of their remaining weights, capped at the amount needed to reach $\alpha n^2-hn$, and add that weight to the bundle. Each nonfinal step exhausts a fractional variable, so this finite procedure terminates. No edge load increases beyond its original packing load.

Apply \eqref{eq:HR} with error $\zeta/(4w_J)$ and expand each integral copy of $J$ into its clique components. The resulting saving is at least
\[
 w_J\left(\alpha n^2-hn-\frac{\zeta n^2}{4w_J}\right).
\]
Quantization lost less than $\zeta n^2/4$ from the original saving. For sufficiently large $n$, the additional term $w_Jhn$ is at most $\zeta n^2/2$. The total loss is at most $\zeta n^2$. Taking a maximum of the thresholds over the finite collection of templates proves the uniform assertion.
\end{proof}

\begin{proof}[Proof of Theorem~\ref{thm:stabilityintro}]
Fix $\varepsilon>0$, and choose $L,\delta_0,\eta_0$ from Theorem~\ref{thm:fracstability}. The cutoff $L$ is now fixed. Apply Lemma~\ref{lem:rounding} with error $\eta_0n^2/2$. Since $q_L(G)\ge\cp(G)$, the hypotheses on the sequence imply, for all sufficiently large indices,
\[
 q_L^*(G_j)\ge(1/6-\eta_0)n_j^2,
 \qquad\rsd(G_j)\le\delta_0n_j.
\]
Theorem~\ref{thm:fracstability} supplies a clique with the required errors at accuracy $\varepsilon$. Minimize, over the finitely many cliques of each graph, the maximum of the two normalized errors. These minima tend to zero, giving the asserted sequence of cliques. This argument never applies fixed-cutoff rounding with a varying cutoff.

For \eqref{eq:allgraph}, use $L=4$ in Lemmas~\ref{lem:localize} and \ref{lem:rounding}. Their uniform remainder defines a function $\epsilon(n)\to0$ with the claimed bound.
\end{proof}

\section{Compatible integral partitions}\label{sec:construction}

We first construct partitions for an arbitrary division of the vertices into a core and an exterior. The core need not induce a complete graph. Its actual edge count will determine the final bound.

\begin{lemma}\label{lem:saturate}
Let $V(G)=C\sqcup R$, with $|C|=c\ge3$ and $|R|=b$. Suppose every vertex of $C$ misses at most $\tau$ vertices of $R$. If
\begin{equation}\label{eq:saturate}
 b\ge3\lceil c/2\rceil-3+4\tau,
\end{equation}
then there are $e(G[C])$ edge-disjoint triangles, each consisting of one edge of $G[C]$ and one vertex of $R$.
\end{lemma}
\begin{proof}
Put $a=\lceil c/2\rceil$ and divide $C$ into two parts whose orders differ by at most one. For every pair $uv$ of core vertices, its list of allowable colours is $N_R(u)\cap N_R(v)$. At most $2\tau$ vertices of $R$ are absent from this list.

Choose disjoint palettes in $R$ of sizes $a+2\tau$ and $2a-3+2\tau$. On the abstract complete bipartite graph between the two halves of $C$, every list in the first palette has size at least $a$, its maximum degree. Galvin's theorem \cite[Theorem~4.1]{Galvin} gives a proper list edge-colouring. Within each half, use the second palette. Every list has size at least $2a-3$, whereas an edge has at most $2a-4$ adjacent edges, so sequential greedy colouring succeeds. Colours can be reused between the two halves, whose vertex sets are disjoint. The palettes for the crossing and internal pairs are disjoint.

We have properly coloured the pairs of an abstract complete graph on $C$. Retain only the colours of the actual edges of $G[C]$, and extend each such edge through its assigned exterior vertex. Properness prevents any cross edge from occurring twice. Distinct triangles use distinct core edges.
\end{proof}

For a partition $V(G)=C\sqcup R$, use the notation
\begin{gather}
 c=|C|,\quad b=|R|,\quad h=e(G[C]),\quad m=e(G[R]),\quad
 D=cb-e_G(C,R),\label{eq:cutnotation}\\
 \sigma=\max_{x\in R}|C\setminus N(x)|,
 \qquad \tau=\max_{u\in C}|R\setminus N(u)|.\label{eq:cutdeficiencies}
\end{gather}
A maximum over an empty set is zero. In all applications of the following lemma, both sides are nonempty.

\begin{lemma}\label{lem:absorption}
Suppose $c\ge3$, and $G[R]$ has a proper edge-colouring with $k$ colours such that
\begin{align}
 c&\le k\le2(c-2\sigma),\label{eq:capacityone}\\
 b&\ge3\lceil c/2\rceil-3+4\tau+8\left\lceil\frac mk\right\rceil.
 \label{eq:capacitytwo}
\end{align}
Then
\begin{equation}\label{eq:absorption}
 q_3(G)\le cb-D-h-
 \left(\frac{2(c-2\sigma)}k-1\right)m.
\end{equation}
\end{lemma}
\begin{proof}
By Lemma~\ref{lem:colouring}, we can balance the colour classes so that each has size at most $\lceil m/k\rceil$. Inject the $c$ core vertices uniformly into the $k$ colours. A row edge has at least $c-2\sigma$ common core neighbours. For each such neighbour, the probability that its assigned colour equals the edge's colour is $1/k$, and these events are disjoint. Thus some injection retains at least
\[
 t\ge\frac{m(c-2\sigma)}k
\]
row edges whose assigned core vertex is adjacent to both endpoints. Extend these edges to triangles through the assigned vertices. Each colour class is a matching, so these triangles are edge-disjoint. A core vertex spends at most $2\lceil m/k\rceil$ cross edges.

Delete this first triangle family. No core edge has been used, and the remaining column deficiency is at most
\[
 \tau'=\tau+2\left\lceil\frac mk\right\rceil.
\]
Equation~\eqref{eq:capacitytwo} permits Lemma~\ref{lem:saturate} to extend every actual core edge to a triangle in the residual graph. These triangles are disjoint from the first family. Take all remaining edges singly. Since $e(G)=cb-D+h+m$, the number of parts is
\begin{align}
 e(G)-2(h+t)
 &=cb-D-h+m-2t\nonumber\\
 &\le cb-D-h-
 \left(\frac{2(c-2\sigma)}k-1\right)m.
 \label{eq:exactcount}
\end{align}
No elimination property is asserted for the edge-deleted residual.
\end{proof}

The coefficient of $m$ in \eqref{eq:absorption} is positive when the second inequality in \eqref{eq:capacityone} is strict. This strictness will determine the equality graphs.

\begin{lemma}\label{lem:exceptions}
Let $V(G)=C\sqcup R\sqcup T$, where $C$ is a clique, $c=|C|\ge3$, and $t=|T|$. Define the parameters in \eqref{eq:cutnotation}--\eqref{eq:cutdeficiencies} on $G[C\cup R]$. Suppose $G[R]$ has a proper $k$-edge-colouring and \eqref{eq:capacityone}--\eqref{eq:capacitytwo} hold after replacing $\sigma,\tau$ by $\sigma+t,\tau+t$. For $x\in T$, put
\[
 u_x=|N(x)\cap C|,\qquad v_x=|N(x)\cap R|.
\]
Then
\begin{equation}\label{eq:exceptions}
 \cp(G)\le cb-\binom c2+
 \sum_{x\in T}|v_x-u_x|+4t^2.
\end{equation}
\end{lemma}
\begin{proof}
For each $x\in T$, let $B_x$ be the bipartite graph between
$N(x)\cap C$ and $N(x)\cap R$ whose edges are the corresponding cross
edges of $G$. Choose matchings $M_x\subseteq E(B_x)$, pairwise disjoint as
sets of cross edges, so that $\sum_{x\in T}|M_x|$ is maximum. Such a family
exists because the graph is finite. For a fixed $x$, the matching $M_x$ is
maximum after the edges used by the other matchings have been deleted;
otherwise replacing $M_x$ would increase the total size.

Put $q_x=|M_x|$ and
\[
 L_x=\min\{u_x,v_x\}-q_x.
\]
In the available bipartite graph for $x$, choose $L_x$ unmatched vertices
on each side. No available edge joins the two chosen sets, since $M_x$ is
maximum. At any one vertex, the other matchings use at most $t$ cross
edges, one from each matching. The rectangle between the two chosen sets
therefore contains at most $tL_x$ edges of $G$, while it has $L_x^2$
pairs. Since $D$ is the total number of missing pairs between $C$ and $R$,
\[
 L_x^2-tL_x\le D.
\]
If $L_x\le t$ then $L_x\le\sqrt D+t$. If $L_x>t$, the preceding inequality
gives $(L_x-t)^2\le L_x(L_x-t)\le D$. Hence in all cases
\begin{equation}\label{eq:matchingloss}
 L_x\le\sqrt D+t,
 \qquad
 \sum_{x\in T}L_x\le t\sqrt D+t^2.
\end{equation}

For every edge $cr\in M_x$, use the triangle $xcr$. These triangles are
edge-disjoint: each $M_x$ is a matching, and the matchings are disjoint in
their cross edges. Remove all their cross edges from $G[C\cup R]$. At any
vertex of $C\cup R$, at most one edge from each $M_x$ has been removed, so
the maximum row and column deficiencies increase by at most $t$. The
internal graphs on $C$ and $R$ are unchanged. Since $c\ge3$, the hypotheses
permit Lemma~\ref{lem:absorption} to be applied to the residual graph. We
drop its nonnegative residual-edge saving. If $q=\sum_{x\in T}q_x$, then
the residual missing-cross count is $D+q$.

Add the $q$ exceptional triangles, take every unused edge from $T$ to
$C\cup R$ singly, and also take all edges of $G[T]$ singly. Since $C$ is a
clique, the number of parts is at most
\begin{align*}
 cb-\binom c2-D
 &+\sum_{x\in T}(u_x+v_x-2q_x)+e(G[T])\\
 &=cb-\binom c2-D
   +\sum_{x\in T}|v_x-u_x|+2\sum_{x\in T}L_x+e(G[T]).
\end{align*}
Using \eqref{eq:matchingloss}, the inequality
$-D+2t\sqrt D\le t^2$, and $e(G[T])\le\binom t2$, the error beyond the
claimed main terms is at most
\[
 t^2+2t^2+\binom t2\le4t^2.
\]
\end{proof}

\section{Near-extremal structure and exactness}\label{sec:rigidity}

Fix $s\ge0$. In this section all asymptotic notation refers to a sequence with $n\to\infty$, while $s$ is held fixed. We use the constants
\begin{equation}\label{eq:smallconstants}
 \lambda=\frac1{100(s+1)},\qquad \rho=\frac{\lambda}{100}.
\end{equation}
They satisfy
\[
 8\rho<\lambda/3,
 \qquad 4\lambda<1/6,
 \qquad (2s+2)\lambda=1/50.
\]

\subsection{A bounded exceptional set}

\begin{lemma}\label{lem:normalform}
Suppose $\rsd(G)\le s$, $\delta(G)\ge n/3-o(n)$, and $G$ has a clique $A$ with
\[
 |A|=n/3+o(n),\qquad D_A+e(G-A)=o(n^2).
\]
For all sufficiently large orders there is a partition
$V(G)=C\sqcup R\sqcup T$, with $C$ a clique and $|T|\le T_0(s)$, such that
\begin{gather}
 |C|=n/3+o(n),\qquad D+e(G[R])=o(n^2),
 \qquad\omega(G[R])=o(n),\label{eq:normalmass}\\
 \sigma\le2\rho n,\qquad \tau\le\lambda n,
 \qquad\Delta(G[R])\le |R|-\lambda n/2.
 \label{eq:normaluniform}
\end{gather}
Here $D,\sigma,\tau$ refer to $C,R$. The bound $T_0(s)$ is independent of the sequence and of the rates in its error terms.
\end{lemma}
\begin{proof}
Remove from $A$ every vertex missing more than $\rho n/4$ original exterior vertices. There are $o(n)$ such vertices. The remaining clique $A_0$ has order $n/3+o(n)$, and its exterior $R_0$ has $o(n^2)$ edges. Hence $\omega(G[R_0])=o(n)$. Every vertex of $A_0$ misses at most $\rho n/4$ vertices of $R_0$, since the moved core vertices are adjacent to all of $A_0$.

Let $Y_0$ consist of the vertices of $R_0$ with degree at least $\rho n$ in $G[R_0]$. Consider the bipartite graph of nonedges between $A_0$ and $Y_0$. Its matching number is less than
\[
 \ell=\left\lceil\frac{4(s+1)}\rho\right\rceil
\]
for large $n$. Suppose instead that it contains $\ell$ disjoint missing pairs. Each pair has at least $\rho n/2$ common neighbours in $R_0$. If $j_y$ counts how many pairs have both endpoints adjacent to $y$, then
\[
 \sum_{y\in R_0}j_y\ge2(s+1)n,
 \qquad 0\le j_y\le\ell.
\]
If $M$ vertices have $j_y\ge s+1$, the sum is at most $sn+\ell M$. Thus $M\ge n/\ell$. Counting the $(s+1)$-subsets of pairs seen by these vertices, some fixed $s+1$ pairs have at least
\[
 \frac{n}{\ell\binom{\ell}{s+1}}
\]
common neighbours in $R_0$. Every induced subgraph of $G[R_0]$ is colourable with at most $\omega(G[R_0])+s=o(n)$ colours, by Lemma~\ref{lem:rooted}. This linear-size common neighbourhood contains $s+2$ independent vertices for large $n$, contradicting Lemma~\ref{lem:forbidden}.

Take a maximal matching of the nonedge graph. Put its fewer than $\ell$ row endpoints in $T_1$, and move its fewer than $\ell$ core endpoints into the exterior. Call the remaining core $A_1$ and exterior $R_1$. Every retained vertex of $Y_0$ is complete to $A_1$, as is every moved core vertex. We still have
\[
 |A_1|=n/3+o(n),\qquad e(G[R_1])=o(n^2),
 \qquad\omega(G[R_1])=o(n).
\]
A retained row not originally in $Y_0$ has degree less than $\rho n+O_s(1)$ in $R_1$.

Put
\[
 U=\{x\in R_1:d_{R_1}(x)>|R_1|-\lambda n\}.
\]
The sparse edge count gives $|U|=o(n)$. Every member of $U$ is complete to $A_1$, by the preceding low-degree bound. Each vertex of $A_1\cup U$ misses at most $\lambda n$ vertices of $R_1$, counting a vertex itself as missed when it lies in $R_1$.

The complement of $G[A_1\cup U]$ has no matching of size $s+1$. Indeed, the endpoints of such a matching would have at least
\[
 |R_1|-(2s+2)\lambda n=\Omega(n)
\]
common neighbours in $R_1$. Its clique number is $o(n)$, so Lemma~\ref{lem:rooted} again supplies $s+2$ independent common neighbours, contradicting Lemma~\ref{lem:forbidden}.

Let $Z_0$ be the endpoint set of a maximal matching in this complement. Then $|Z_0|\le2s$, and
\[
 C=(A_1\cup U)\setminus Z_0,
 \qquad R=R_1\setminus U,
 \qquad T=T_1\cup Z_0
\]
is a partition with $C$ a clique and $|T|\le\ell-1+2s$. In particular, vertices removed from the promoted core go into $T$, not back into $R$.

The changes from $A$ involve $o(n)$ vertices, so the cross and exterior errors remain $o(n^2)$, and $\omega(G[R])=o(n)$. A retained low-degree row has at most $\rho n+O_s(1)$ neighbours outside $A_1\cup T_1$. Minimum degree therefore makes it adjacent to all but $\rho n+o(n)$ vertices of $A_1$. The other retained rows are complete to $A_1$. Since $|U|=o(n)$, this gives $\sigma\le2\rho n$ for large $n$. Core columns miss at most $\lambda n$ rows by construction. Finally, every row outside $U$ has degree at most $|R_1|-\lambda n$, and $|U|=o(n)$ gives
\[
 \Delta(G[R])\le |R|-\lambda n/2.
\]
\end{proof}

\subsection{Incidences of exceptional neighbourhoods}

\begin{lemma}\label{lem:incidence}
Let $C,R,X$ be disjoint vertex sets, with $C$ a clique and
$\rsd(G[C\cup R\cup X])\le s$. Put
\[
 h=|X|,\quad w=\omega(G[R]),\quad
 a=\max_{r\in R}|C\setminus N(r)|.
\]
If every $x\in X$ satisfies
\begin{equation}\label{eq:largeholes}
 |C\setminus N(x)|>a+w+h,
\end{equation}
then
\begin{equation}\label{eq:incidence}
 e(X,R)\le s|R|+h(s+w).
\end{equation}
\end{lemma}
\begin{proof}
Take an $s$-defective order ending in $C$. At the deletion of $r\in R$, a forward-neighbourhood clique containing $x\in X$ has order at most
\[
 |N_C(x)|+w+h<|C|-a.
\]
The retained core neighbours of $r$ form a clique of order at least $|C|-a$. Hence a maximum forward-neighbourhood clique avoids $X$. Its deficiency is therefore at least the number of later $X$-neighbours of $r$, which is at most $s$.

At the deletion of $x\in X$, a forward-neighbourhood clique has order at most
\[
 |N_C(x)|+w+d_X^+(x).
\]
The forward deficiency bound gives $d_R^+(x)\le s+w$. Count each edge between $X$ and $R$ at its earlier endpoint to obtain \eqref{eq:incidence}.
\end{proof}

\begin{lemma}\label{lem:rectangle}
Let $(x_i,y_i)\in[0,1]\times[0,2]$ for $1\le i\le t$, and put
\[
 I=\{i:y_i>x_i,\ x_i<1\},\qquad h=|I|.
\]
If $\sum_{i\in I}y_i\le2s$, then
\begin{equation}\label{eq:rectangle}
 \sum_{i=1}^t|y_i-x_i|\le t+s-|h-s|.
\end{equation}
If the left side is at least $t+s$, then exactly $s$ points are $(0,2)$ and every other point is $(1,0)$ or $(1,2)$.
\end{lemma}
\begin{proof}
Outside $I$, each summand is at most one. Inside $I$, it is at most $y_i\le2$. Therefore the sum is at most
\[
 t-h+\min\{2h,2s\}=t+s-|h-s|.
\]
Equality at $t+s$ forces $h=s$. It then forces $x_i=0,y_i=2$ on $I$, and $|y_i-x_i|=1$ outside $I$. The latter points, by the definition of $I$, are precisely $(1,0)$ and $(1,2)$.
\end{proof}

\subsection{Exact rigidity}

\begin{proposition}\label{prop:rigidity}
Fix $s\ge0$. Suppose $|G_j|=n_j\to\infty$, $\rsd(G_j)\le s$,
\[
 \delta(G_j)\ge n_j/3-o(n_j),
 \qquad \cp(G_j)\ge Q_s(n_j).
\]
Then all sufficiently large members of the sequence have the form \eqref{eq:extremal}, and their clique partition numbers equal $Q_s(n_j)$.
\end{proposition}
\begin{proof}
Suppose an infinite subsequence consists of graphs not having the asserted form. Theorem~\ref{thm:stabilityintro} supplies the cliques needed for Lemma~\ref{lem:normalform}, since $Q_s(n)=n^2/6+O_s(n)$. Use its partition $C\sqcup R\sqcup T$, and suppress sequence subscripts. Pass to a subsequence on which $t=|T|$ is constant. Label these boundedly many vertices and pass again so that every pair
\[
 \left(\frac{3u_i}{n},\frac{3v_i}{n}\right)
 \longrightarrow(x_i,y_i),
 \qquad u_i=|N_C(t_i^{(n)})|,\quad v_i=|N_R(t_i^{(n)})|,
\]
converges. Here $t_i^{(n)}$ denotes the $i$th exceptional vertex; the limits are real numbers in $[0,1]\times[0,2]$.

The row graph has degeneracy at most $\omega(G[R])+s-1$, so Lemma~\ref{lem:colouring} gives a proper edge-colouring with
\[
 k=\max\{c,\Delta(G[R])+\omega(G[R])+s\}
\]
colours. The parameters of the normal form imply
\[
 \frac{k}{n}\le\frac23-\frac\lambda2+o(1),
 \qquad
 \frac{2(c-2(\sigma+t))}{n}\ge\frac23-8\rho-o(1).
\]
Since $c=n/3+o(n)$, we have $c\ge3$ for all sufficiently large indices. Since $8\rho<\lambda/2$, the first capacity condition of Lemma~\ref{lem:exceptions} holds. For the second, $k\ge c=n/3+o(n)$ and $m=o(n^2)$ give
\[
 \frac{b-3\lceil c/2\rceil+3-4(\tau+t)-8\lceil m/k\rceil}{n}
 \ge\frac16-4\lambda-o(1)>0.
\]
Thus \eqref{eq:exceptions} applies.

Its base term is at most $M_{n-t}$, by \eqref{eq:quadratic} with $s=0$, whereas
\[
 Q_s(n)-M_{n-t}=\frac{s+t}{3}n+O_{s,t}(1).
\]
Divide \eqref{eq:exceptions} by $n/3$ after subtracting this base bound, and pass to the limit. We obtain
\begin{equation}\label{eq:profilelower}
 \sum_{i=1}^t|y_i-x_i|\ge s+t.
\end{equation}

Let $I$ be as in Lemma~\ref{lem:rectangle}, and let $X$ consist of the corresponding exceptional vertices in each graph. Each misses a positive linear number of vertices of $C$. Put
\[
 \eta_n=\sqrt{D/n^2}+n^{-1/2}.
\]
Delete from $R$ every row missing more than $\eta_n n$ core vertices. Their number is at most $D/(\eta_n n)=o(n)$. On the remaining rows the maximum core deficiency and the clique number are both $o(n)$. For large $n$, condition \eqref{eq:largeholes} holds for $X$, whose size is bounded. Lemma~\ref{lem:incidence} applies. Restoring the omitted rows changes $e(X,R)$ by $o(n)$, so division by $n/3$ gives
\[
 \sum_{i\in I}y_i\le2s.
\]
Together with \eqref{eq:profilelower}, Lemma~\ref{lem:rectangle} forces exactly $s$ limits to be $(0,2)$. Denote their vertex set by $X$, and denote the sets with limits $(1,0)$ and $(1,2)$ by $P$ and $Z$, respectively. Thus
\begin{equation}\label{eq:threeprofiles}
\begin{array}{c|cc}
 &|N_C(v)|&|N_R(v)|\\ \hline
 v\in X&o(n)&b-o(n)\\
 v\in P&c-o(n)&o(n)\\
 v\in Z&c-o(n)&b-o(n).
\end{array}
\end{equation}

The set $C\cup Z$ is a clique for all sufficiently large indices. Otherwise choose a nonedge in it. For each of the $s$ vertices $x_i\in X$, choose a distinct missed vertex $a_i\in C$, avoiding the first nonedge's endpoints. Equation~\eqref{eq:threeprofiles} gives linearly many choices. These are $s+1$ vertex-disjoint nonedges. Their endpoints have a common neighbourhood in $R$ of order at least
\[
 b-(s+2)\lambda n-o(n)=\Omega(n):
\]
old core columns miss at most $\lambda n$ rows, while $X$ and $Z$ miss $o(n)$. Since $\omega(G[R])+s=o(n)$, that common neighbourhood has $s+2$ independent vertices. This contradicts Lemma~\ref{lem:forbidden}. The same reasoning covers $s=0$, when the first nonedge alone is forbidden.

Now use a fresh partition of the original graph, with
\[
 C'=C\cup Z\cup X,\qquad R'=R\cup P.
\]
The set $C'\setminus X=C\cup Z$ is a clique of order $|C'|-s$. The preliminary matchings used to prove the numerical bound \eqref{eq:exceptions} are not retained. Denote the parameters of this new cut by primes. Equations \eqref{eq:normalmass}--\eqref{eq:normaluniform} and \eqref{eq:threeprofiles} give
\begin{gather}
 c'=n/3+o(n),\quad b'=2n/3+o(n),\quad
 \sigma'\le2\rho n+o(n),\quad \tau'\le\lambda n+o(n),\label{eq:primedeficiencies}\\
 m'=o(n^2),\quad \omega(G[R'])=o(n),\quad
 \Delta(G[R'])\le b'-\lambda n/3.
 \label{eq:primemass}
\end{gather}
For the degree assertion, old rows gain only boundedly many neighbours, while every added row in $P$ has $o(n)$ row neighbours. The other assertions follow by moving the bounded set $T$ and using its profiles.

Colour the new row graph with
\[
 k'=\max\{c',\Delta(G[R'])+\omega(G[R'])+s\}
\]
colours. Then
\[
 k'/n\le2/3-\lambda/3+o(1),\qquad
 2(c'-2\sigma')/n\ge2/3-8\rho-o(1).
\]
The first capacity inequality is strict because $8\rho<\lambda/3$. The second has normalized slack at least $1/6-4\lambda-o(1)>0$, by \eqref{eq:primemass}. Lemma~\ref{lem:absorption} yields, with $\theta=2(c'-2\sigma')/k'-1>0$,
\begin{align}
 \cp(G)
 &\le c'b'-e(G[C'])-D'-\theta m'\nonumber\\
 &\le c'(n-c')-\binom{c'-s}{2}-D'-\theta m'\nonumber\\
 &\le Q_s(n).
 \label{eq:rigidcount}
\end{align}
Here \eqref{eq:quadratic} gives the last inequality, including its integer maximum.

Our hypothesis forces equality throughout this finite count. Therefore $D'=m'=0$ and $e(G[C'])=\binom{c'-s}{2}$. The known clique $C'\setminus X$ accounts for every core edge; the $s$ vertices of $X$ are isolated within the core. All cross edges are present and the exterior is independent. Equality in the quadratic forces the allowed nearest-integer choices of $c'$. This contradicts the choice of the subsequence and proves the proposition.
\end{proof}

\subsection{The extremal theorem}

\begin{proof}[Proof of Theorem~\ref{thm:main}]
Fix $s$. We first prove the constant-error bound. Suppose $\cp(G)-Q_s(|G|)$ is unbounded on graphs with $\rsd(G)\le s$. For integers $K\to\infty$, choose a graph $G_K$ of smallest order violating
\[
 \cp(G_K)\le Q_s(|G_K|)+K.
\]
The orders tend to infinity, since there are only finitely many graphs of any bounded order. Every one-vertex deletion has defect at most $s$ and satisfies the bound with this same $K$. Covering the incident edges singly gives, for every vertex $v$,
\[
 \cp(G_K)\le Q_s(n-1)+K+d(v).
\]
Since the values are integral,
\[
 d(v)\ge Q_s(n)-Q_s(n-1)+1
 =\left\lfloor\frac{n+s+1}{3}\right\rfloor+1>n/3.
\]
Proposition~\ref{prop:rigidity} applies and gives $\cp(G_K)=Q_s(n)$ for all sufficiently large $K$, a contradiction. The difference is therefore bounded above, proving \eqref{eq:constantbound} after increasing $K_s$ to a nonnegative integer if necessary.

We now use this established bound to prove eventual exactness. If $\cp(G)\ge Q_s(n)$, apply \eqref{eq:constantbound} to $G-v$. It gives
\[
 d(v)\ge Q_s(n)-Q_s(n-1)-K_s
 \ge n/3-K_s-1.
\]
Any unbounded sequence of such graphs satisfies Proposition~\ref{prop:rigidity}. Hence, above a threshold depending on $s$, all are the graphs in \eqref{eq:extremal} and have value $Q_s(n)$. All other graphs have smaller value. Lemma~\ref{lem:attainment} supplies the attaining examples at every sufficiently large order.
\end{proof}

\begin{proof}[Proof of Corollary~\ref{cor:chordal}]
By Lemma~\ref{lem:rooted}, chordal graphs are exactly those of rooted defect zero. Substitute $s=0$ in Theorem~\ref{thm:main}.
\end{proof}

\begin{corollary}\label{cor:editstability}
For every fixed $s\ge0$ and every $\varepsilon>0$, there are $\delta>0$ and $N$ such that if $n\ge N$, $\rsd(G)\le s$, and
\[
 \cp(G)\ge Q_s(n)-\delta n^2,
\]
then at most $\varepsilon n^2$ edge additions and deletions turn $G$ into a graph of the form \eqref{eq:extremal}.
\end{corollary}
\begin{proof}
It suffices to prove the sequential assertion with $\delta=o(1)$. Otherwise choose counterexamples with $\delta=1/j$ and order at least $j$. Theorem~\ref{thm:stabilityintro} produces a clique $A$ of order $n/3+o(n)$ such that adding its $D_A$ missing cross pairs and deleting the $e(G-A)$ exterior edges gives a book after $o(n^2)$ edits.

Choose an allowed $k$ from \eqref{eq:extremal}. Both $k-s$ and $|A|$ are $n/3+o(n)$. Move $o(n)$ vertices between the two sides to obtain a clique core of order $k-s$ and an independent exterior; changing every adjacency incident with those vertices costs $o(n^2)$. Choose $s$ exterior vertices, delete their edges to the core, and join them to all the other exterior vertices. This costs $O_s(n)$ further edits and gives \eqref{eq:extremal}.
\end{proof}

\section{An exact finite signed bound}\label{sec:integer}

The extremal profile also has a finite interpretation that needs only one defective elimination order. Define
\begin{equation}\label{eq:profileP}
 P_s(n)=\max_{0\le k\le\lfloor n/2\rfloor}
 \left\{k(n-k)-\binom{\max\{k-s,1\}}2\right\}.
\end{equation}
The values at orders zero and one are zero.

\begin{lemma}\label{lem:increments}
For nonnegative integers $n,s$,
\begin{equation}\label{eq:profileclosed}
 P_s(n)=
 \begin{cases}
 \lfloor n^2/4\rfloor,&n<2s,\\
 s(n-s)+F_{n-2s},&n\ge2s.
 \end{cases}
\end{equation}
For $n\ge1$,
\begin{equation}\label{eq:increments}
 P_s(n)-P_s(n-1)=
 \min\left\{\left\lfloor\frac n2\right\rfloor,
              \left\lfloor\frac{n+s+1}{3}\right\rfloor\right\}.
\end{equation}
In particular, $P_s(n)=Q_s(n)$ for $n\ge2s$.
\end{lemma}
\begin{proof}
When $n<2s$, every admissible $k$ is at most $s$, so the binomial term is zero and the maximum is $\lfloor n^2/4\rfloor$. Suppose $n\ge2s$. Candidates with $k<s$ do not exceed the candidate $k=s$. Put $k=s+j$ and $m=n-2s$. For $0\le j\le\lfloor m/2\rfloor$, the objective is
\[
 s(n-s)+j(m-j)-\binom j2.
\]
The last two terms have maximum $F_m$: for $m\ge2$ a nearest integer to $(2m+1)/6$ lies in this interval, so \eqref{eq:quadratic} applies, and $m=0,1$ are immediate. This proves \eqref{eq:profileclosed}.

Successive differences of $\lfloor n^2/4\rfloor$ are $\lfloor n/2\rfloor$, and those of $F_n$ are $\lfloor(n+1)/3\rfloor$. Thus for $n>2s$ the difference is $s+\lfloor(n-2s+1)/3\rfloor=\lfloor(n+s+1)/3\rfloor$, which is no larger than $\lfloor n/2\rfloor$. For $n\le2s$ the other term is minimal, and the boundary $n=2s$ gives the same value $s$. Finally
\[
 F_{n+s}-F_{n-2s}=s(n-s)+\binom{s+1}{2},
\]
because the difference of the corresponding real quadratics is this integer. This gives $P_s(n)=Q_s(n)$ in the asserted range.
\end{proof}

\Needspace{10\baselineskip}
\begin{theorem}\label{thm:integer}
Suppose $G$ has $n$ vertices and admits an $s$-defective elimination order. If integer edge weights satisfy
\begin{equation}\label{eq:integerconstraints}
 z_e\le1,\qquad
 \sum_{e\in E(T)}z_e\le1\quad\text{for every triangle }T,
\end{equation}
then
\begin{equation}\label{eq:integerbound}
 \sum_{e\in E(G)}z_e\le P_s(n).
\end{equation}
For every $n,s$, the bound is attained by a graph with $\rsd(G)\le s$ and a signing that satisfies every clique inequality, not just the triangle inequalities.
\end{theorem}
\begin{proof}
The assertion is immediate when $n=0$, so assume $n\ge1$. Raise every coordinate below $-1$ to $-1$. This preserves \eqref{eq:integerconstraints}: an affected triangle has one coordinate equal to $-1$ and the other two at most one. The objective can only increase. We may assume all coordinates belong to $\{-1,0,1\}$.

Write $\sigma_z(v)=\sum_{e\ni v}z_e$ and $\mu=\min_v\sigma_z(v)$. The positive-edge graph is triangle-free. At the endpoints of a positive edge, the positive neighbourhoods are disjoint, so one endpoint has positive degree at most $\lfloor n/2\rfloor$. Its signed degree is no larger. If there is no positive edge, all signed degrees are nonpositive. Hence
\begin{equation}\label{eq:signedhalf}
 \mu\le\lfloor n/2\rfloor.
\end{equation}
This is the needed bound when $n\le2s+2$.

Suppose $n\ge2s+3$. If $\mu\le s$, then $\mu\le\lfloor(n+s+1)/3\rfloor$. Otherwise choose the first vertex $v$ of an $s$-defective order, and a largest clique $Q\subseteq N(v)$. Let $a_-,a_+$ count the negative and positive spokes from $v$ to $Q$, and let $r_-,r_+$ count them outside $Q$. Then
\begin{equation}\label{eq:signedlocal}
 r_-+r_+\le s,\qquad
 \mu\le a_++r_+-a_- -r_-.
\end{equation}
Since $\mu>s$, there is a positive neighbour $u\in Q$. Its edges to the other $a_+-1$ positive neighbours of $v$ in $Q$ have weight $-1$, by the triangle constraints. Its positive neighbourhood is disjoint from that of $v$, which has order $a_++r_+$. Therefore
\[
 \mu\le\sigma_z(u)\le n-(a_++r_+)-(a_+-1).
\]
Combining this with \eqref{eq:signedlocal} gives
\begin{align*}
 n&\ge\mu+2a_++r_+-1\\
  &\ge3\mu+2a_-+2r_- -r_+-1\\
  &\ge3\mu-s-1.
\end{align*}
The signed degrees are integers, so $\mu\le\lfloor(n+s+1)/3\rfloor$. Together with \eqref{eq:signedhalf}, this is the increment in \eqref{eq:increments}.

Delete a minimum-signed-degree vertex. The induced graph retains an $s$-defective order, and the signing remains feasible. Induction, starting with the empty graph, now proves \eqref{eq:integerbound} by telescoping \eqref{eq:increments}.

For sharpness, suppose $n\ge2$ and choose a maximizing $k\ge1$ in \eqref{eq:profileP}. Put $a=\max\{k-s,1\}$, and take
\[
 G=(K_a\sqcup\ind_{k-a})\vee\ind_{n-k}.
\]
The page-first argument of Lemma~\ref{lem:attainment} gives $\rsd(G)\le k-a\le s$. Give core-clique edges weight $-1$ and all spokes weight $1$. A clique containing a page and $j$ core vertices has weight $j-\binom j2\le1$; a clique contained in the core has nonpositive weight. The total signing value is $k(n-k)-\binom a2=P_s(n)$. Empty and one-vertex graphs attain zero.
\end{proof}

\begin{corollary}\label{cor:maxcut}
If $G$ admits an $s$-defective elimination order, then
\[
 2\MaxCut(G)-e(G)\le P_s(|G|).
\]
The bound is sharp at every order, even on the class $\rsd(G)\le s$.
\end{corollary}
\begin{proof}
Give crossing edges of a maximum cut weight $1$ and internal edges weight $-1$. Every triangle has zero or two crossing edges, so \eqref{eq:integerconstraints} holds. The sharpness construction in Theorem~\ref{thm:integer} uses this signing on its core--page cut.
\end{proof}

The integer signed theorem is not a pointwise partition formula. For $G=K_3\vee\ind_2$, the core-negative, spoke-positive signing has value three. Assigning weight $1/2$ to each of the six page triangles gives an exact fractional partition of value three, which also bounds every signed clique functional. But $\cp(G)=4$. A four-clique and the three remaining spokes give four parts. Three parts are impossible: a four-clique leaves three separate spokes, while without a four-clique all nine edges would have to form three triangles, contrary to the odd degrees of the two pages. Thus the finite signed result does not replace the integral construction in Theorem~\ref{thm:main}.

\section{Open problems}

The attaining graphs in Theorem~\ref{thm:integer} also have clique partition number $P_s(n)$: there are at least as many pages as core vertices, so the modular edge-colouring construction attains the page-edge lower bound. Theorem~\ref{thm:main} proves the matching upper bound eventually for each fixed $s$. The exact finite problem remains.

\begin{conjecture}
For all nonnegative integers $n,s$,
\[
 \max\{\cp(G):|G|=n,\ \rsd(G)\le s\}=P_s(n).
\]
\end{conjecture}

The eventual equality classification need not hold at small orders. For example, $\cp(P_4)=3=F_4$, but $P_4$ is not a book.

A quantitative version of Corollary~\ref{cor:editstability} would relate the number of adjacency changes needed to reach \eqref{eq:extremal} to
\[
 Q_s(n)-\cp(G).
\]
Determine the optimal dependence on this deficit, on $n$, and on $s$. The proof above gives qualitative stability, but does not give a sharp rate or a useful numerical threshold for eventual exactness.

\section*{Acknowledgments}

The author thanks Dimitrios Los for carefully reviewing an earlier version of this manuscript.

\end{document}